\documentclass[11pt,a4paper]{article}
\usepackage{amsmath,amssymb,amsthm,amsfonts,latexsym,graphicx,subfigure}
\usepackage[all,import]{xy}
\usepackage{indentfirst,cite}
\usepackage{tabularx,booktabs}
\usepackage{caption,amscd}
\usepackage{hyperref}

\usepackage{fancyhdr,enumerate}
\usepackage{bbm,color}
\usepackage{tikz}
\usetikzlibrary{shapes.geometric, arrows}
\usepackage{accents,cases}
\usepackage{multirow}
\usepackage{euscript}\usepackage{wasysym}\usepackage{pdfsync}\usepackage{comment}
\newtheorem{introthm}{Theorem}

\usepackage{array,float}
\newcommand{\PreserveBackslash}[1]{\let\temp=\\#1\let\\=\temp}
\newcolumntype{C}[1]{>{\PreserveBackslash\centering}p{#1}}
\newcolumntype{R}[1]{>{\PreserveBackslash\raggedleft}p{#1}}
\newcolumntype{L}[1]{>{\PreserveBackslash\raggedright}p{#1}}

\DeclareMathOperator*{\cone}{\ensuremath{cone}}

\DeclareMathOperator*{\conv}{\ensuremath{conv}}

\newcommand{\h}{\mathcal{H}}
\newcommand{\dist}{\mathrm{dist}}

\newcommand{\K}{\ensuremath{\mathcal{K}}}

\newcommand{\E}{\ensuremath{\mathcal{E}}}
\newcommand{\C}{\ensuremath{\mathcal{C}}}
\newcommand{\B}{\ensuremath{\mathcal{B}}}
\newcommand{\T}{\ensuremath{\mathcal{T}}}
\makeatletter
\def\wbar{\accentset{{\cc@style\underline{\mskip8mu}}}}

\makeatother

\renewcommand{\vec}[1]{\mbox{\boldmath \small $#1$}}

\newcommand{\power}{\ensuremath{\mathcal{P}}}

\newcommand{\R}{\ensuremath{\mathbb{R}}}
\newcommand{\A}{\ensuremath{\mathcal{A}}}

\theoremstyle{plain}
\newtheorem{theorem}{Theorem}[section]

\newtheorem{defn}{Definition}[section]
\newtheorem{notification}{Convention}
\newtheorem{lemma}{Lemma}[section]
\newtheorem{remark}{Remark}

\newtheorem{pro}{Proposition}[section]
\newtheorem{example}{Example}[section]
\newtheorem{Question}{Question}[section]

\def\D{{\mathcal D}}
\def\Lov{Lov\'asz extension }

\allowdisplaybreaks[4]
\begin{document}
\bibliographystyle{unsrt}
\title{Discrete-to-Continuous Extensions:  Lov\'asz extension and Morse theory
}
\author{J\"urgen Jost\footnotemark[1], \and Dong Zhang\footnotemark[2]}
\footnotetext[1]{Max Planck Institute for Mathematics in the Sciences, Inselstrasse 22, 04103 Leipzig,
Germany. \\Email address:
{\tt  jost@mis.mpg.de}  (J\"urgen Jost).
}
\footnotetext[2]{Max Planck Institute for Mathematics in the Sciences, Inselstrasse 22, 04103 Leipzig, Germany. \\
Email addresses:
{\tt dzhang@mis.mpg.de} \; and \; {\tt 13699289001@163.com} (Dong Zhang).
}
\date{}
\maketitle

\begin{abstract}

This is the first of a series of papers that  develop a systematic
  bridge between constructions in  discrete mathematics and the corresponding   continuous analogs. 
 In this paper, we establish an equivalence between Forman's discrete Morse theory on a simplicial complex and the continuous Morse theory (in the sense of any known non-smooth Morse theory) on the associated  order complex via the    Lov\'asz extension. Furthermore, we propose a new version of the  Lusternik-Schnirelman category on abstract simplicial complexes  to bridge the classical Lusternik-Schnirelman theorem and its discrete analog on finite complexes. More generally,
we can suggest a discrete Morse theory on hypergraphs by employing piecewise-linear (PL) Morse theory and Lov\'asz extension, hoping to provide new tools for exploring  the structure of hypergraphs.

\vspace{0.2cm}

\end{abstract}
\tableofcontents

\tikzstyle{startstop} = [rectangle, rounded corners, minimum width=1cm, minimum height=1cm,text centered, draw=black, fill=red!0]
\tikzstyle{io1} = [rectangle, trapezium left angle=80, trapezium right angle=100, minimum width=1cm, minimum height=1cm, text centered, draw=black, fill=blue!0]
\tikzstyle{io2} = [trapezium,  rounded corners, trapezium left angle=100, trapezium right angle=100, minimum width=1cm, minimum height=1cm, text centered, draw=black, fill=yellow!0]
\tikzstyle{process} = [rectangle, minimum width=1cm, minimum height=1cm, text centered, draw=black, fill=orange!0]
\tikzstyle{decision} = [circle, minimum width=1cm, minimum height=1cm, text centered, draw=black, fill=green!0]
\tikzstyle{decision2} = [ellipse, rounded corners=10mm, minimum width=2cm, minimum height=2cm, text centered, draw=black, fill=green!0]
\tikzstyle{arrow} = [thick,->,>=stealth]

\section{Introduction and Background}\label{sec:introduction}

The Lov\'asz extension is a basic tool in discrete mathematics, especially for some combinatorial optimization problems and submodular analysis \cite{Lovasz}. It was introduced in the study of submodular functions which appear frequently in many areas like game theory, matroid theory, stochastic processes, electrical networks, computer vision and machine learning \cite{F05-book}.

\vspace{0.16cm}

In fact, a special form of the Lov\'asz extension appeared already in the context of the
 Choquet integral \cite{Choquet54} which has fruitful applications in statistical mechanics, potential theory and  decision theory.
Since the Lov\'asz extension does not require
the monotonicity of the set function in finite cases of the Choquet integral, it has a wider range of applications, for instance in  
 combinatorics, for algorithms  in computer science.  

\vspace{0.16cm}

We shall start by looking  at the original Lov\'asz extension. For simplicity, we shall work throughout this paper with a finite and nonempty set $V=\{1,\cdots,n\}$ and its power set $\mathcal{P}(V)$. 
We denote the cardinality of a set $A$ by $\#A$.
Given a function $f:\mathcal{P}(V)\to \R$, one identifies every $A\in \mathcal{P}(V)$ with its indicator vector $\vec1_A\in \{0,1\}^n\subset\R^n $. The Lov\'asz extension  extends the domain of $f$ to the whole Euclidean space\footnote{Some other versions in the literatures only extend the domain to the cube $[0,1]^V$ or the nonnegative orthant  $\R_{\ge0}^V$. In fact, many works on Boolean lattices  identify  $\power(V)$ with the discrete cube $\{0,1\}^n$.} $\R^V$ in the following way: 

For $\vec x =(x_1,\dots ,x_n)\in \mathbb{R}^n$, let $\sigma:V\cup\{0\}\to V\cup\{0\}$ be a bijection such that $ x_{\sigma(1)}\le x_{\sigma(2)} \le \cdots\le x_{\sigma(n)}$ and $\sigma(0)=0$, where $x_0:=0$. The Lov\'asz extension of $f$ is defined by
\begin{equation}\label{eq:Lovasum}
f^{L}(\vec x)=\sum_{i=0}^{n-1}(x_{\sigma(i+1)}-x_{\sigma(i)})f(V^{\sigma(i)}(\vec x)),
\end{equation}
where   $V^0(\vec x)=V$ and $V^{\sigma(i)}(\vec x):=\{j\in V: x_{j}> x_{\sigma(i)}\},\;\;\;\; i=1,\cdots,n-1$.

It is known that $f^L$ is positively one-homogeneous, PL (piecewise linear) and Lipschitzian continuous \cite{Lovasz,Bach13}. Also, $f^L(\vec x+t\vec 1_V)=f^L(\vec x)+tf(V)$, $\forall t\in\R$, $\forall \vec x\in\R^V$. Moreover, a continuous function $F:\R^V\to \R$ is the Lov\'asz extension of some $f:\power(V)\to\R$ if and only if $F(\vec x+\vec y)=F(\vec x)+F(\vec y)$ whenever $(x_i-x_j)(y_i-y_j)\ge0$, $\forall i,j\in V$.


\vspace{0.13cm}


It is an important feature of the \Lov that submodular functions are
  extended to convex ones and therefore, discrete optimization problems
  involving submodular functions become amenable to the powerful tools of
  convex optimization. It is our systematic purpose to explore the power of
  the \Lov in
the context of functions that are not submodular or convex and explore its
potential there. In this paper, which is a part of our general program, we want to use the \Lov to study the interplay between  discrete and continuous aspects in  Morse theory.
 To this end, we shall work on a restricted version of the  Lov\'asz extension. 
For $\A\subset\power(V)$ and $f:\A\to\R$, we take  $\D_\A=\{\vec
x\in\R^{n}_{\ge0}:V^t(\vec x)\in\A,\forall t\in\R\}$ as a feasible domain of the  Lov\'asz extension $f^L$,  where $V^t(\vec x)=\{i\in V:  x_i>t\}$ is the upper level set of $\vec x$ at level $t$ (using this notation, we indeed write $V^{x_{\sigma(i)}}(\vec x)$  as $V^{\sigma(i)}(\vec x)$ in  \eqref{eq:Lovasum} for simplicity). It is clear that for $\vec x\in \D_\A$, the  Lov\'asz extension $f^L$ of $\vec x$  is well-defined by \eqref{eq:Lovasum}.




The discrete Morse theory on simplicial complexes was introduced by Forman
\cite{Forman98,Forman02}. This theory has some deep  connections with smooth
Morse theory \cite{Benedetti12,Gallais10,Benedetti16}, as well as practical
applications \cite{RWS11}, and also admits several slight generalizations. It
also possesses surprising applications in algebraic combinatorics and  derived
algebraic geometry  \cite{AroneBrantner}. Both this  discrete Morse theory
and the classical smooth one make assumptions that exclude some
complicated cases such as  monkey-saddle points, which makes them  simpler.

We will construct the relationship between the Morse theory of a discrete
Morse function, that is, a function defined on the face set  $\K$ of an abstract 
  simplicial complex, and its Lov\'asz extension in Section
\ref{sec:discrete-Morse}. We define the \Lov on the order complex  $S_\K$, the simplicial
  complex whose vertex set is  
  $\K$ and whose faces are the
 inclusion  chains in $\K$. In this way, we can use the procedure of \Lov to extend a
  function on a set of discrete points, the vertices of $S_\K$.

We restrict the Lov\'asz extension $f^L$ to a geometric realization of the order complex $|S_\K|$ which is a subset of the feasible domain of $f^L$. It is surprising that the Lov\'asz  extension on restricted domains leads to  a fascinating connection between discrete and continuous Morse theory and Lusternik-Schnirelman theory:  

\begin{introthm}[Theorems \ref{thm:critical-discrete-Morse}, \ref{thm:Morse-vector-f} and \ref{thm:LS-category-discrete-Morse-Lovasz}]
\label{Mainthm:discrete-Morse}
For a simplicial complex with vertex set $V$ and face set $\K$, let $f:\K\to\R$ be an injective discrete Morse function. Then the following conditions are equivalent:
\begin{enumerate}[(1)]
\item $\sigma$ is a critical point of $f$   with $\dim \sigma=i$;
\item $\vec1_\sigma$ is a critical point of
  $f^L|_{|S_\K|}$ with index $ i$ in the sense of weak slope (metric Morse theory);
\item  $\vec1_\sigma$ is a critical point of $f^L|_{|S_\K|}$ with index $ i$ in the sense of K\"uhnel (PL Morse theory);
\item  $\vec1_\sigma$ is a Morse critical point of $f^L|_{|S_\K|}$ with index $ i$ in the sense of topological Morse theory.
\end{enumerate}
Here the notation $|S_\K|$ indicates a suitable restriction, more
  precisely the  geometric realization of the order complex of $\K$ whose
  vertices are the simplices\footnote{Here and after, for convenience, we do not differentiate  
 between  a  simplicial complex and its face set (both are denoted by $\K$).} of $\K$ and whose simplices correspond to the
  chains in $\K$,  (see Section \ref{sec:discrete-Morse}) for $f^L$ being well-defined.

Moreover, the discrete Morse vector $(n_0,n_1,\cdots,n_d)$,  representing the number $n_i$ of critical points with index $i$, of $f$ coincides with the  continuous Morse vector of $f^L|_{|S_\K|}$.

Moreover, the Lusternik-Schnirelmann category\footnote{The definition of category classes for $\K$ is introduced in \eqref{eq:LSC-K} in  Section \ref{sec:discrete-Morse}. } theorem is preserved  under Lov\'asz extension:
$$
\min\limits_{L\in\mathsf{LSC}_m(\K)}\max\limits_{\sigma\in L} f(\sigma)=\inf\limits_{S\in \mathsf{LSC}_m(|S_\K|)}\sup\limits_{\vec x\in S} f^L(\vec x),
$$
where we refer to \eqref{eq:LSC-K} and  \eqref{eq:LSC-S_K} for the definition of the  Lusternik-Schnirelmann classes $\mathsf{LSC}_m(\K)$ and $\mathsf{LSC}_m(|S_\K|)$, 
respectively. 
\end{introthm}

   
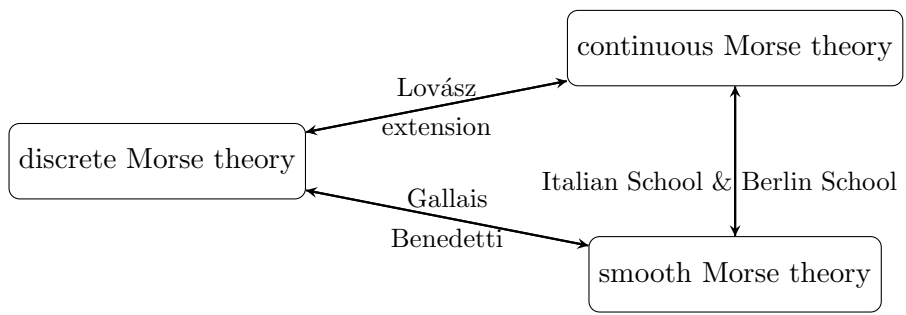
\begin{center}
\begin{tikzpicture}[node distance=6cm]

\node (disMorse) [startstop] { discrete Morse theory };

\node (Morse) [startstop, right of=disMorse, xshift=1.6cm, yshift=1.5cm]  { continuous Morse theory }; 

\node (smoothMorse) [startstop, right of=disMorse, xshift=1.6cm, yshift=-1.5cm] { smooth Morse theory };

\draw [arrow](disMorse) --node[anchor=south] { \small Lov\'asz } (Morse);
\draw [arrow](Morse) --node[anchor=north] {\small extension  } (disMorse);
\draw [arrow](smoothMorse) --node[anchor=north] {\small Benedetti  } (disMorse);
\draw [arrow](disMorse) --node[anchor=south] { \small Gallais } (smoothMorse);
\draw [arrow](Morse) --node[anchor=north] {\small  Italian School \& Berlin School \;\;\;} (smoothMorse);
\draw [arrow](smoothMorse) --node[anchor=north] {\small } (Morse);
\end{tikzpicture}
\end{center}

In summary, Theorem \ref{Mainthm:discrete-Morse} says that the {\sl Morse structures} of $\K$ and $|S_\K|$ are coarsely equivalent, and one can translate all the results about  `Morse data' of a discrete Morse function $f$ on $\K$ to its Lov\'asz extension $f^L$ restricted on $|S_\K|$. This also reflects the deep result from \cite{Gallais10,Benedetti16} that smooth Morse theory on a manifold is almost equivalent to the discrete Morse theory on its triangulation. The difference is that  we don't assume the complex $|\K|$ to be a topological manifold, so that   topological results on manifolds cannot be applied directly. Fortunately, our feasible domain $|S_\K|$ is a piecewise flat geometric complex. Our proofs don't draw heavily on the standard tools in discrete Morse theory. 


The idea above allows us to establish a discrete Morse theory on hypergraphs, which helps us to understand the structure of a hypergraph from a Morse theoretical perspective.  We  borrow the ideas on  associated
simplicial complexes of hypergraphs  \cite{PL91} and order complexes  induced by   hypergraphs \cite{Wachs04,Stanley79}, as well as  finite topologies on hypergraphs. With the help of associated
simplicial complexes, we introduce the geometric realizations of  hypergraphs, and we prove that the geometric realization of a hypergraph  collapses onto the geometric realization of  its order complex.  As a preliminary exploration along this direction, we provide some evidence to show that we should focus on the order complex of a hypergraph which looks like a simplicial complex.  Some results similar to Theorem \ref{Mainthm:discrete-Morse} on such a complex-like hypergraph are presented  in Section  \ref{sec:Morse-hypergraph}. 

\vspace{0.19cm}

\textbf{Related works}: 
 The recent paper \cite{RWWW18}
on  discrete Morse theory for hypergraphs   uses the so-called embedded homology on hypergraphs  \cite{BLRW19}, that is,  they embed the hypergraph into a simplicial complex obtained by adding some missing simplices, and define  a Morse function on a hypergraph as the restriction of a Morse function on that simplicial complex.  The underlying homology theory \cite{GWXW21} takes the difference between the original hypergraph and the embedding complex into account. 
The theory of \cite{RWWW18}  is different from our approach. We consider both  associated
simplicial complexes and order complexes from the perspective of homotopy, and our definition of Morse functions on  hypergraphs uses order complexes. Such an order complex is naturally defined from the hypergraph, without the need to add any further simplices. In particular, the order complex of a hypergraph in general is different from that of an embedding simplicial complex.

We  provide a dictionary between different analogs of Morse theory.  The relations between different versions of  Morse theory  have great potential to translate a
problem in one context to another, thereby giving new tools for attacking problems and drawing  connections. The key tool 
is the Lov\'asz extension, with which we succeeded  in
constructing  fruitful  relations between certain  discrete objects 
and their continuous analogs  \cite{JostZhang2,JostZhang3}. 

\section{Preliminaries on Morse theory}
\label{sec:critical}

Morse theory \cite{M36,M37} enables us to analyze the topology of an object $M$ by studying  functions $f:M\to\R$. 
In the classical case, $M$ is a manifold and $f$ is generic and differentiable. 
There are, however, many extensions of Morse theory  in modern mathematics that do not   require a smooth structure,  such as
the metric and topological Morse theory by the Italian school \cite{DM94,IS96,K94,D10}, the PS (piecewise smooth) or stratified Morse theory by Thom, Goresky and MacPherson \cite{GM88}, the PL  Morse theory by Banchoff \cite{Banchoff67}, K\"uhnel \cite{Brehm-Kuhnel87,Edelsbrunner-Harer10} and the Berlin school, as well as the discrete Morse theory by Forman \cite{Forman98,Forman02}. 

In all such cases, a typical function $f$ on $M$ will reflect the topology quite directly,  allowing one to find CW structures 
on $M$ and to obtain  information about their homology.  The following results embody the abstract content of Morse theory, and they  hold in continuous as well as in discrete cases.

\vspace{0.19cm}

\textbf{Morse fundamental theorem}. If $f$ has $n_i$ critical points of index $i$, $i=0,1,\cdots,d$, then $M$ is
	homotopy equivalent to a cell complex (called Morse complex) with $n_i$ cells of dimension $i$. One can write it as
$$M\simeq \text{cell complex with }n_i\text{ cells of }\dim i$$

\textbf{Morse relation}. Denote by $P(X,A)(\cdot)$ the Poincare polynomial\footnote{Formally, $P(X,A)(t):=\sum_{n\ge0}\mathrm{rank}\, H^n(X,A)t^n$, where $H^n(X,A)$ is the relative cohomology of the pair $(X,A)$.} of the pair of topological spaces $(X,A)$ over a given field $\mathbb{F}$, where $X\supset A$. Then
$$
\sum\limits_{a<f(x)<b}P(\{f\le f(x)\},\{f\le f(x)\}\setminus\{x\})(t)=P(\{f<b\},\{f\le a\})(t)+(1+t)Q(t)
$$
where $a<b$, $Q(\cdot)$ is a polynomial with nonnegative coefficients.

The main aim of this section is to study the Lov\'asz extension of a discrete Morse function on a simplicial complex, and to provide equivalences between discrete Morse theory and its Lov\'asz extension.

For this purpose, we first clarify the notions and  concepts and summarize the various Morse theories mentioned above.

\begin{enumerate}
\item[-] \textbf{Metric Morse theory}:
Let $M$ be a metric space and  $F$  a continuous function on $M$. For a point $\vec a\in M$,
there exists $\epsilon\ge0$ such that there exist $\delta>0$ and a continuous map
$$\h:B_\delta(\vec a)\times[0,\delta]\to M$$
satisfying $$F(\h(\vec x,t))\le F(\vec x)-\epsilon t,\;\; \dist(\h(\vec x,t),\vec x)\le t$$
for any $\vec x\in B_\delta(\vec a)$ and $t\in[0,\delta]$,  where $ B_\delta(\vec a)$ is an open ball in $M$.  The {\sl weak slope} \cite{DM94,IS96,K94} denoted by $|dF|(\vec a)$ is defined as the supremum of such $\epsilon$ above. A point $\vec a$ is called a critical point of $F$ on $M$, if it has vanishing weak slope, i.e., $|dF|(\vec a)=0$.

The local behaviour of $F$ near $\vec a$ is described by the so-called {\sl
  critical group} $C_q(F,\vec a):=H_q(\{F\le c\}\cap U_{\vec a},\{F\le c\}\cap
U_{\vec a}\setminus \{\vec a\})$, $q\in\mathbb{Z}$, where $H_*(\cdot,\cdot)$
is the singular relative homology  with real field coefficients, and $U_{\vec a}$ is an open neighborhood of $\vec a$.  So the Morse polynomial $p(F,\vec
a)(t):=\sum_{q=0}^d \mathrm{rank}\, C_q(F,\vec a) t^q$ can be defined. If
$C_q(F,\vec a)$ is non-vanishing, then we say that $q$ is an
index of a metric critical point $\vec a$, and the number
$p(F,\vec a)(1)$ is called the {\sl total multiplicity} of $\vec a$. (Note that, in general, a critical point may have more
    than one index. The standard assumptions of Morse theory, however, exclude
    that possibility.)

\item[-] \textbf{Topological Morse theory}:
Let $M$ be a topological space and  $F$  a continuous function on $M$. A point $\vec a\in M$ is a {\sl Morse regular point} of $F$ if there exist a neighborhood $U$  of $\vec a$ in $M$  and a continuous map $$\h:U\times[0,1]\to M,\;\; \h(\vec x,0)=\vec x$$
satisfying $$F(\h(\vec x,t))< F(\vec x),$$
for any $\vec x\in U$ and $t>0$. We say that $\vec a$ is a {\sl Morse critical point} of  $F$ on $M$ if it is not Morse regular. The index with multiplicity of a critical point is same as in the metric setting above \cite{D10}.

A {\sl symmetric homological critical value} \cite{C-SEH07} of $F$ is a real number $c$ for which there exists an integer such that for all sufficiently small $\epsilon> 0$, the map $H_k(\{F\le c-\epsilon\})\hookrightarrow H_k(\{F\le c+\epsilon\})$ induced by inclusion is not an isomorphism \cite{B-Scott14}. Here $H_k$ denotes the $k$-th singular homology (possibly with coefficients in a field).

A  real number $c$ is a {\sl homological regular value} of the function $F$ if there exists $\epsilon> 0$ such that for each pair of real numbers $t_1<t_2$ on the interval $(c-\epsilon,c+\epsilon)$, the inclusion $\{F\le t_1\}\hookrightarrow \{F\le t_2\}$ induces isomorphisms on all homology groups \cite{B-Scott14}. A
real number that is not a homological regular value of $F$ is called a {\sl homological
critical value} of $F$.

\item[-] \textbf{Piecewise-Linear Morse theory}: Similar to the smooth setting, the PL (piecewise linear) Morse theory introduced by Banchoff  requires working with a {\sl combinatorial manifold} which is both a PL manifold and a simplicial complex. Here we will use the notions developed by K\"uhnel \cite{Brehm-Kuhnel87} and later by Edelsbrunner  and Harer  \cite{Edelsbrunner-Harer10}. 

    Denote by $\mathrm{star}_-(v)$ the subset of the star of $v$ on which the PL function $F$ takes values not greater than $F(v)$. Similarly, one can define $\mathrm{link}_-(v)$. 

Let $M$ be a combinatorial manifold, and let $F$ be a PL (piecewise linear) function on $M$.
\begin{defn}[K\"uhnel \cite{Brehm-Kuhnel87}]
A vertex $v$ of $M$ is said to be a PL critical point of $F$ with index $i$ and multiplicity $k_i$ if $\beta_i(\overline{\mathrm{star}_-(v)},\mathrm{link}_-(v))=k_i$, where $\beta_i$ is 
the $i$-th Betti number of the relative homology group.
\end{defn}
Equivalently, let $\beta'_j$ be the rank of the reduced $j$-th homology group of $\mathrm{link}_-(v)$. Using this notation, we have
\begin{defn}[Edelsbrunner \& Harer  \cite{Edelsbrunner-Harer10}]
A vertex $v$ is a PL critical point of $F$ with index $i$ and multiplicity $k_i$ if $\beta'_{i-1}=k_i$.
\end{defn}
Clearly, a PL critical point may have many indices and multiplicities. A vertex $v$ is called non-degenerate critical if its total multiplicity $\sum_{i=0}^d k_i$ is equal to $1$. The PL  function $F$ is called a PL Morse function if all critical vertices are non-degenerate.

\item[-] \textbf{Discrete Morse theory}: A discrete Morse function on an abstract simplicial complex $(V,\K)$ is a function $f:\K\to\R$ satisfying for any $p$-dimensional simplex $\sigma\in \K_p$, $\#U(\sigma)\le1$ and $\#L(\sigma)\le1$, where
$$U(\sigma):=\{\tau^{p+1}\supset\sigma:f(\tau)\le f(\sigma)\}\;\text{ and }\;L(\sigma):=\{\nu^{p-1}\subset\sigma:f(\nu)\ge f(\sigma)\}.$$
\begin{defn}[Forman \cite{Forman98,Forman02}]
We say that $\sigma\in \K_p$ is a critical point of $f$ on $\K$ if
$\#U(\sigma)=0$ and $\#L(\sigma)=0$. The index of a critical point $\sigma$ is defined to be $\dim\sigma$.
\end{defn}
\end{enumerate}

The main results in this paper  can be summarized by:
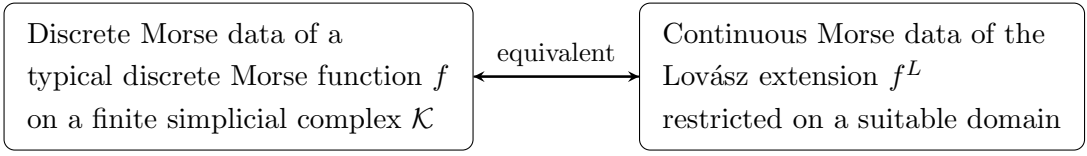
\begin{center}
\begin{tikzpicture}[node distance=6.6cm]

\node (convex) [startstop] { $\begin{array}{l}\text{Discrete Morse data of a} \\ \text{typical discrete Morse function }f \\\text{on a finite simplicial complex }\K  \end{array}$};

\node (submodular) [startstop, right of=convex, xshift=1.6cm]  { $\begin{array}{l}\text{Continuous Morse data of the} \\ \text{Lov\'asz extension }f^L \\\text{restricted on a suitable domain}  \end{array}$};

\draw [arrow](convex) --node[anchor=south] { \small equivalent } (submodular);
\draw [arrow](submodular) --node[anchor=south] {   } (convex);
\end{tikzpicture}
\end{center}

 While the discrete Morse data are taken here in the sense of Forman,  the continuous Morse data can be in the  metric, topological or PL category as described above.


Precise statements are presented in the following section.

\label{sec:weak-slope-Morse}

\section{Relations between discrete Morse theory and its Lov\'asz  extension}
\label{sec:discrete-Morse}


We let $\K$ be an abstract  simplicial complex with vertex set $V$, and we do not distinguish between $\K$ and its face set.   Let $\K_p$ be the collection of $p$-simplices (or $p$-dimensional faces)  in $\K$.
\begin{defn}\label{defn:ordercomplex}
The {\sl order complex} of $\K$ is defined by
$$S_\K:=\{\C\subset\K: \C\text{ is a chain}\},$$
where $\C$ is a {\sl chain} if for any $\sigma_1,\sigma_2\in\C$, either $\sigma_1\subset\sigma_2$ or $\sigma_2\subset\sigma_1$. 
 It is clear that the order complex  $S_\K$  is itself a simplicial complex with the vertex set
$\K$. Define the  geometric realization of $S_\K$ by
$$|S_\K|= \bigcup_{\C\in S_\K}\conv(\vec 1_{\sigma}:\sigma\in \C),$$
where $\conv$  denotes the convex hull.
\end{defn}

\noindent \textbf{Fact}: For any function $f:\K\to\R$, the feasible domain $\D_\K$ of its Lov\'asz extension $f^L$ is $\bigcup\limits_{t\ge0}t|S_\K|$. This means that the Lov\'asz extension $f^L$ is well-defined on $|S_\K|$.

We refer to Proposition \ref{pro:setfamily-order-complex} for the proof of the  fact above  and the observation below. 

\vspace{0.13cm}

\noindent\textbf{Observation}:
\begin{equation*}
|S_\K|=\D_\K\cap S_\infty
= \bigcup_{\text{maximal chain }\C\subset\K}\conv(\vec 1_{\sigma}:\sigma\in \C),
\end{equation*}
where $S_\infty=\{\vec x\in\R^n:\|\vec x\|_\infty=1\}$ is the unit $l^\infty$-sphere, and $\conv$ denotes convex hull. 
Maximal chains from $\K$  correspond to 
facets (that is, simplices not contained in the others) 
of $|S_\K|$.

\begin{lemma}\label{lemma:Morse-sigma}
Given a discrete Morse function $f$ on a finite simplicial complex $\K$, we have:
\begin{enumerate}[(1)]
\item If $\sigma$ is critical, then  $f(\tau)> f(\sigma)> f(\nu)$, whenever $\tau\supsetneqq\sigma\supsetneqq \nu$.

\item If $f$ is an injective Morse function and $(\sigma,\tau)$ is a regular pair (i.e., $\sigma\subset\tau$ with $\dim \tau=\dim \sigma+1$ and $f(\sigma)> f(\tau)$), then
\begin{enumerate}[({2}.1)]
\item for all $\tau'\supsetneqq \sigma$ with $\tau'\not\supset \tau\setminus\sigma$, $f(\tau')>f(\sigma)$;

\item for any $\sigma'\subsetneqq\tau$ with $\sigma'\supset \tau\setminus\sigma$, $f(\sigma')<f(\tau)$;

\item  for each $\sigma''\subsetneqq\sigma$, $f(\sigma'')<f(\sigma)$.
\end{enumerate}
\end{enumerate}
\end{lemma}

\begin{proof}[Proof of Lemma \ref{lemma:Morse-sigma}] Let $c=f(\sigma)$.
Note that $f(\nu^{p-1})<c$ for all $\nu^{p-1}\subset \sigma$. If there exists $\nu^{p-2}\subset \sigma$ such that $f(\nu^{p-2})\ge c$, then $f(\nu^{p-2})\ge f(\sigma)>f(\nu^{p-1})$ for all $\nu^{p-1}\supset\nu^{p-2}$ with $\nu^{p-1}\subset \sigma$. Since there are two $\nu^{p-1}$ in $\sigma$ containing $\nu^{p-2}$, this is not compatible with  the definition of a discrete Morse function. In this way, we can prove  by  induction on the dimension of faces of $\sigma$  that every face $\nu\subset \sigma$ satisfies $f(\nu)<c$.

The  proofs of the other statements  are similar.
\end{proof}

\begin{notification}
We use $\cong$ and $\simeq $ to express  homeomorphism equivalence and homotopy equivalence, respectively. The link and star of some $\sigma\in\K$ will be taken on $S_\K$. The operation $*$  denotes the geometric join operator \cite{Brown}.
\end{notification}

\begin{lemma}\label{lemma:link-homotopy}
Given an injective  discrete Morse function, we have:
$$
\mathrm{link}_-(\sigma)\simeq\begin{cases}\mathbb{S}^{\dim \sigma-1} ,& \text{ if } \sigma\text{ is critical},\\
\mathrm{pt},& \text{ if }  \sigma\text{ is regular}.
\end{cases}
$$
\end{lemma}
\begin{proof}
The link of $\sigma$ in the order complex $|S_\K|$ is the geometric join\footnote{The (geometric) join of  subsets $A$ and $B$ in $\R^n$ is defined as  $A*B:=\{t\vec a+(1-t)\vec b:\vec a\in A,\vec b\in B,0\le t\le1\}$. We refer to \cite{Brown} for details.} of
$$\mathbb{S}_-(\vec1_\sigma):=\bigcup_{\text{ chain }\C\subset\mathcal{P}(\sigma)\setminus\{\sigma\}}\conv(\vec 1_{\nu}:\nu\in \C)\cong \mathbb{S}^{\dim\sigma-1}$$
and
$$\bigcup_{\text{ chain }\C\subset\left\{\tau\in\K\left|\tau\supsetneqq\sigma\right.\right\}}\conv(\vec 1_{\tau}:\tau\in \C).$$
According to Lemma \ref{lemma:Morse-sigma} and the definition of $\mathrm{link}_-(\sigma)$, we obtain that if $\sigma$ is critical, then $$\mathrm{link}_-(\sigma):=\mathrm{link}_-(\vec1_\sigma)=\mathbb{S}_-(\vec1_\sigma) \cong \mathbb{S}^{\dim\sigma-1}.$$

If $(\sigma,\tau)$ is a regular pair, we note that $\mathrm{link}_-(\sigma)$ is the join of
$\mathbb{S}_-(\vec1_\sigma)$ and
$$
\bigcup_{\text{ chain }\C\in [\sigma,\tau]_f}\conv(\vec 1_{\tau'}:\tau'\in \C)\simeq \vec 1_{\tau},
$$
where $[\sigma,\tau]_f:=\{\tau'\supset\tau: f(\tau')<f(\sigma)\}$. That means, $\mathrm{link}_-(\sigma)\simeq \mathbb{S}^{\dim\sigma-1}*\vec 1_{\tau}\cong\mathbb{B}^{\dim\sigma}\simeq \mathrm{pt}$.

Similarly, one can check that
$\mathrm{link}_-(\tau)\cong \mathbb{B}^{\dim\tau-1}\simeq \mathrm{pt}$. 
 The proof is completed.
\end{proof}

\begin{lemma}[K\"uhnel \cite{Kuhnel90}] \label{lemma:Kuhnel}
Given a real-valued PL function $f^{PL}$ on a simplicial complex $|\K|$, then the induced subcomplex of $\K$ on $\{v\in \K_0: f^{PL}(v)\le t\}$ is homotopy-equivalent to the sublevel set $\{f^{PL}\le t\}$.
\end{lemma}

\begin{lemma}\label{lemma:sub-level-critial}
Given an injective discrete Morse function, denote by $\epsilon_0= \min\{|f(\sigma)-f(\sigma')|:\sigma\ne \sigma'\}>0$.

If $\sigma$ is critical, then
$$ \{f^L\le t\}\cap \mathrm{star}(\vec1_\sigma)\simeq\begin{cases}\mathbb{S}^{\dim \sigma-1} ,& \text{ if }  f(\sigma)-\epsilon_0<t< f(\sigma),\\
\mathbb{B}^{\dim \sigma},& \text{ if }  f(\sigma)\le t<f(\sigma)+\epsilon_0,
\end{cases}
$$
where $\mathrm{star}(\vec1_\sigma)$ is the star of 
$\vec 1_\sigma$ in  $|S_\K|$.  
%
In consequence,   $\vec1_\sigma$ is a topological/metric critical point of $f^L|_{|S_\K|}$, and $f(\sigma)$ is a (symmetric) homological critical value.
\end{lemma}

\begin{proof}
 Denote by
$$|S_\sigma|= \bigcup_{\text{maximal chain }\C\subset\mathcal{P}(\sigma)}\conv(\vec 1_{\nu}:\nu\in \C).$$
Then it can be checked that $|S_\sigma|$ is 
homeomorphic to the closed geometric simplex $|\overline{\sigma}|$ in $|\K|$, and thus it is homeomorphic to the disc $\mathbb{B}^{\dim \sigma}$. 
Since $\sigma$ is critical, by Lemma \ref{lemma:Morse-sigma} and the definition of $\epsilon_0$, for any $\nu\subsetneqq \sigma$, $f(\nu)\le f(\sigma)-\epsilon_0$. 
Hence, for $f(\sigma)-\epsilon_0<t< f(\sigma)$,  $|S_\sigma|\cap \mathrm{star}(\vec1_\sigma)\cap \{f^L<t\}$ is homotopy-equivalent to  
$$\bigcup_{\text{ chain }\C\subset\mathcal{P}(\sigma)\setminus\{\sigma\}}\conv(\vec 1_{\nu}:\nu\in \C)=\partial |S_\sigma|\cong  \mathbb{S}^{\dim \sigma-1}.$$ 
Similarly, for any $\tau\supsetneqq \sigma$, $f(\tau)\ge f(\sigma)+\epsilon_0>t+\epsilon_0$. Therefore, 
together with the piecewise linearity of $f^L$, one gets that $\mathrm{star}(\vec1_\sigma)\cap \{f^L<t\}$ is homotopy-equivalent to $|S_\sigma|\cap \mathrm{star}(\vec1_\sigma)\cap \{f^L<t\}$ and thus the proof is completed. 

The case of $f(\sigma)\le t<f(\sigma)+\epsilon_0$ is similar. 
For more details, we may apply Lemma \ref{lemma:Kuhnel} to $f^L$ on $|S_\K|$. Then we only need to check the homotopy type of $\mathrm{star}_-(\sigma)$ in $|S_\K|$ for $t\ge f(\sigma)$ and $\mathrm{link}_-(\sigma)$ for $t<f(\sigma)$. According to Lemma \ref{lemma:Morse-sigma} and similar to the proof of Lemma \ref{lemma:link-homotopy}, we obtain that for a critical point $\sigma$, $\mathrm{star}_-(\sigma)$ is
$$\bigcup_{\text{ chain }\C\subset\mathcal{P}(\sigma)}\conv(\vec 1_{\nu}:\nu\in \C)\cong \mathbb{B}^{\dim\sigma},$$
and $\mathrm{link}_-(\sigma)\cong \mathbb{S}^{\dim\sigma-1}$. The proof is completed.
\end{proof}

\begin{lemma}\label{lemma:weak-slope-metric}
If $(\sigma,\tau)$ is a regular pair, then $\left|df^L|_{|S_\K|}\right|(\vec1_\sigma)>0$ and $\left|df^L|_{|S_\K|}\right|(\vec1_\tau)>0$.
\end{lemma}
\begin{proof}
By the definition of weak slope, we should construct a locally decreasing flow from a neighborhood of $\vec1_\sigma$ to a neighborhood of $\vec1_\tau$.
\begin{enumerate}[{Case} 1.]
	\item Locally decreasing flow near $\vec1_\sigma$:\;  
	For any chain containing the pair $(\sigma,\tau)$, 
$f^L(\vec 1_\sigma)=f(\sigma)> f(\tau)=f^L(\vec 1_\tau)$, which means that  $f^L$ is decreasing along the vector $\overrightarrow{\vec 1_\sigma\vec1_\tau}$.  Then with the help of Lemma  \ref{lemma:Morse-sigma} (2), the neighborhood of $\vec1_\sigma$ on $|S_\K|$ 
can be decreased uniformly along the direction $\overrightarrow{\vec 1_\sigma\vec1_\tau}$ with a small  modification. Precisely, we equip  $|S_\K|$ with the shortest path distance `$\mathrm{dist}$' induced by the usual Euclidean metric on $\R^n$.  Then, for sufficiently small  $\delta<\frac19$ and an open ball $ B_\delta(\vec 1_\sigma)$  in $|S_\K|$, we   define the locally decreasing flow 
$$\h:B_\delta(\vec1_\sigma)\times[0,\delta]\to |S_\K|$$
determined by $\h(\vec x,t)=(1-t)\vec x +t\vec1_\tau$ if $\vec x\in B_\delta(\vec 1_\sigma)\cap\mathrm{star}(\vec1_\tau) ,\,t\in[0,\delta]$, and $$\h(\vec x,t)=\vec x+t\frac{\mathrm{proj}(\vec x)-\vec x}{\|\mathrm{proj}(\vec x)-\vec x\|_2}$$ when 
$\vec x\in B_\delta(\vec 1_\sigma) 
\setminus\mathrm{star}(\vec1_\tau)$ and $0\le t\le\min\{\delta,\|\mathrm{proj}(\vec x)-\vec x\|_2\}$, 
 where $\mathrm{proj}(\vec x)$ is the projection of $\vec x$ onto  $\mathrm{star}(\vec1_\tau)\cap \triangle(\vec x)$, and  $\triangle(\vec x)$ is the smallest simplex of $|S_\K|$  containing $\vec x$. We refer to Fig.~\ref{Fig:1} for a   sketch of the picture of the construction of the local flow $\h$.
\begin{figure}[H]
\centering
{\color{black}\begin{tikzpicture}[scale=2]
\draw[thin,yellow] (0,0)--(0,2);
\draw[thin,yellow] (1.732,1)--(0,2);
\draw[thin,yellow] (1.732,1)--(0,0);
\draw[thin,yellow] (-1.732,1)--(0,2);
\draw[thin,yellow] (-1.732,1)--(0,0);
\filldraw[black,fill opacity=0.04,line width=0,draw =red!0]  (0,0)--(0,2)--(1.732,1);
\filldraw[black,fill opacity=0.02,line width=0,draw =red!0]  (0,0)--(0,2)--(-1.732,1);
\node (123) at  (0.58,1.2) {$\vec1_\tau$};\node (123) at  (0.577,1) {{\color{blue}\tiny $\bullet$}};
\node (134) at  (-0.6,1.2) {$\vec1_{\tau'}$};\node (134) at  (-0.577,1) {{\color{blue}\tiny$\bullet$}};
\node (s) at  (0.1,1.2) {$\vec1_{\sigma}$};\node (s) at  (0,1) {{\color{red}\tiny$\bullet$}};
\draw[thin,green] (-1.732,1) -- (1.732,1);
\draw[thin,green] (-0.577,1)--(0,0)--(0.577,1)--(0,2)--(-0.577,1);
\draw[thin,green] (0.866,1.5) -- (0.577,1) -- (0.866,0.5);
\draw[thin,green] (-0.866,1.5) -- (-0.577,1) -- (-0.866,0.5);
\draw[dotted] (0,1)  circle(0.15);
\draw[->,purple] (-0.12,1.1) to (0,1.1);
\draw[->,purple] (-0.12,0.9) to (0,0.9);
\draw[->,purple] (-0.15,1) to (0,1);
\draw[->] (0,1.1) to (0.288,1.05);
\draw[->] (0,0.9) to (0.288,0.95);
\draw[->] (0,1) to (0.288,1);
\end{tikzpicture}   } 
\caption{\label{Fig:1} This picture visually  illustrates the construction of a locally decreasing flow near $\vec1_\sigma$ in  the proof of Lemma \ref{lemma:weak-slope-metric}. In a sufficiently small neighborhood $B_\delta(\vec 1_\sigma)\cap\mathrm{star}(\vec1_\tau)$ of $\vec1_\sigma$, the piecewise linear flow  in  $B_\delta(\vec 1_\sigma)\cap\mathrm{star}(\vec1_\tau)$ goes towards $\vec1_\tau$ (see the black line in the picture), while the flow line in  $B_\delta(\vec 1_\sigma)\setminus\mathrm{star}(\vec1_\tau)$ is orthogonal to $\mathrm{star}(\vec1_\tau)$ (see the red line in the picture). It follows from $f^L(\vec1_{\tau'})>f^L(\vec1_{\sigma})>f^L(\vec1_{\tau})$ 
for any  $\tau'\supsetneqq \sigma$ with $\tau'\not\supset \tau$, that $f^L$ is decreasing along the  local flow.}
\end{figure}
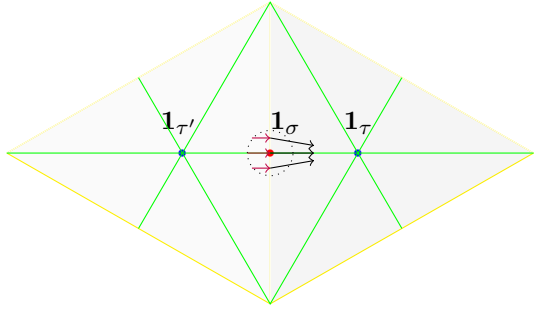

According to Lemma \ref{lemma:Morse-sigma} (2),  $f^L(\vec1_{\tau'})=f(\tau')>f^L(\vec1_{\sigma})=f(\sigma)>f(\tau)=f^L(\vec1_{\tau})$ for any  $\tau'\supsetneqq \sigma$ with $\tau'\not\supset \tau$. Then, one can check that there exists $\epsilon>0$ 
satisfying $$f^L(\h(\vec x,t))\le f^L(\vec x)-\epsilon t,\;\; \dist(\h(\vec x,t),\vec x)\le t$$
for any $\vec x\in B_\delta(\vec 1_\sigma)$ and $t\in[0,\delta]$. We then  complete the construction of a locally decreasing flow near $\vec1_\sigma$.  
	
	
	\item Locally decreasing flow near $\vec1_\tau$:\; The construction  depends on Lemma \ref{lemma:Morse-sigma} (2), as in Case 1. Slight perturbations and concrete approximations in the construction of the locally decreasing flow are necessary, but we omit the tedious and elementary process which is similar to that of Case 1. 
\end{enumerate}

By the deformation lemma, $\vec 1_\sigma$ is Morse regular, and by the piecewise linearity of $f^L$, $\vec 1_\sigma$ is not a critical point in the sense of weak slope. 
Moreover,  points on $|S_\K|$ other  than vertices of $|S_\K|$  cannot be critical points  of $f^L$ if $f$ is injective.
\end{proof}

\begin{theorem}\label{thm:critical-discrete-Morse}
Given a finite simplicial complex with vertex set $V$ and face set $\K$, let $f:\K\to\R$ be a discrete Morse function.

If $\sigma$ is a critical point of $f$, then $\vec1_\sigma$ is a critical point of $f^L|_{|S_\K|}$ with the same index in the sense of topological/metric/PL critical point theory, and the converse holds if $f$ is further assumed to be injective. 
\end{theorem}

\begin{proof}
The proof is a combination of Lemmas \ref{lemma:link-homotopy}, \ref{lemma:sub-level-critial} and \ref{lemma:weak-slope-metric}.
\end{proof}

\begin{defn}
If a generic discrete Morse function $f:\K\to\R$ has  $n_i$ critical points of index $i$, we say that both $f$ and $\K$ have {\sl discrete Morse vector}\footnote{Every discrete Morse function  $f$ corresponds to a unique discrete Morse vector, while  there can be a
variety of discrete Morse vectors on a  simplicial complex $\K$ for different $f$. Similar presentations 
work for a
Lipschitz function on a piecewise flat metric space, that is, we refer to the Morse vector as being of both the  function and the space.} $\vec c=(n_0,n_1,\cdots,n_d)$. Similarly, for a generic Lipschitz function on a piecewise flat metric space $M$ having $n_i$  critical points of index $i$,  we say that both $f$ and $M$ have {\sl Morse vector} $\vec c=(n_0,n_1,\cdots,n_d)$.
\end{defn}


Now we verify that the discrete Morse structure on a simplicial complex is equivalent to the continuous Morse structure on the restricted domain of its Lov\'asz extension. The key idea is to translate it into PL  Morse theory by 
barycentric subdivision. This reveals  
the  relation between the discrete Morse vectors of $\K$ and the  Morse vectors of $|S_\K|$. Such a result is relevant to 
the main results in \cite{Benedetti16}, but we develop it here in a wider context.

\begin{theorem}\label{thm:Morse-vector-f}
Given a finite simplicial complex with vertex set $V$ and face set $\K$, let $f:\K\to\R$ be an injective discrete Morse function. Then the discrete Morse vector of $f$ agrees with the continuous  Morse vector of $f^L|_{|S_\K|}$. In particular, every discrete Morse vector on $\K$ is a Morse vector on $|S_\K|$. 
\end{theorem}

\begin{proof}
The simplicial complex $(\K,S_\K)$ is simplicially equivalent to the simplicial complex obtained by the barycentric subdivision of  $(V,\K)$:
 $$
 \xymatrix{ (\K,S_\K)\ar@{<=>}[rr]^{\text{simplicially}}_{\text{equivalent}} & &\mathrm{sd}(V,\K)\ar@{<=>}[ll] }
 $$
 where $\mathrm{sd}(V,\K)$ is the barycentric subdivision of the complex
 $(V,\K)$. Here two complexes are called {\sl simplicially equivalent} (or
 combinatorially equivalent) if their face posets\footnote{The {\sl face
     poset} of a complex is the set of all of its simplices, ordered by
   inclusion. } are isomorphic as posets. The reason for this
   equivalence is simply that the barycenter of a simplex in $\K$ corresponds
   to $\K$ as a vertex in $S_\K$, and when one simplex is a facet of another,
   their barycenters are connected by an edge in $\mathrm{sd}(V,\K)$.


Thus, one may redefine a discrete function $\hat{f}$ on the vertex set of the barycentric subdivision
\begin{center}
\begin{tikzpicture}[node distance=4.5cm]

\node (convex) [startstop] { $\hat{f}:\mathcal{V}(\mathrm{sd}(\K))\to \R$ };

\node (submodular) [startstop, right of=convex, xshift=1.6cm]  { $f:\K\to \R$ };

\draw [arrow](convex) --node[anchor=south] { \small equivalent } (submodular);
\draw [arrow](submodular) --node[anchor=south] {   } (convex);
\end{tikzpicture}
\end{center}
 via $\hat{f}(v_\sigma)=f(\sigma)$, $\forall \sigma\in \K$, where $\mathcal{V}(\mathrm{sd}(\K))$ is the vertex set of $\mathrm{sd}(\K)$. This is essentially
the viewpoint of  Zaremsky when defining a Bestvina-Brady discrete
Morse function in \cite{Zaremsky}. 

Then the Lov\'asz extension $f^L$ is piecewise-linearly equivalent to the piecewise linear extension $\hat{f}^{PL}$ defined by $$\hat{f}^{PL}\left(\sum_{v\in F} t_vv\right)=\sum_{v\in F}t_v\hat{f}(v)$$
 for any face $F$ of the simplicial  complex obtained by the   
 barycentric subdivision of $\K$  
 and any $t_v\ge 0$ with $\sum_{v\in F}t_v=1$.
Combining the above observations, we get the following commutative diagram:
$$
\xymatrix{ f\ar@{<=>}[rrr]\ar@{=>}[d]_{\text{Lov\'asz extension}} & & & \hat{f}\ar@{<=>}[lll]\ar@{=>}[d]^{\text{PL extension}} \\
f^L|_{|S_\K|}\ar@{<=>}[rrr]_{\text{PL equivalent}} & & & \hat{f}^{PL}\ar@{<=>}[lll] }
$$
from which we derive that the Morse data of $f^L|_{|S_\K|}$ and $\hat{f}^{PL}$ are entirely equivalent, and furthermore, the (continuous) Morse structures of $|S_\K|$ and $|\mathrm{sd}(\K)|$ essentially agree with each other.

It is clear that $\{f^L|_{|S_\K|}\le t\}$ is homeomorphic to $\{\hat{f}^{PL}\le t\}$.
Applying Lemma \ref{lemma:Kuhnel}, $\{\hat{f}^{PL}\le t\}$ is  homotopy-equivalent to the induced subcomplex on the sublevel set $\{\hat{f}\le t\}$. Note that the level subcomplex induced by $\{f\le t\}$  collapse onto the induced subcomplex on the sublevel set $\{\hat{f}\le t\}$. So we have
$$|\K(\{f\le t\})|\simeq |\mathrm{Sd}_\K(\{\hat{f}\le t\})|\simeq \{\hat{f}^{PL}\le t\}\cong \{f^L|_{|S_\K|}\le t\}$$
 and thus the statement is proved. 
\end{proof}

Theorems \ref{thm:critical-discrete-Morse} and \ref{thm:Morse-vector-f} establish a correspondence between the geometric data of a discrete Morse function and  the geometric information of its Lov\'asz extension. 


\vspace{0.19cm}

As a special homotopy invariant, the Lusternik-Schnirelmann category was created in order to provide estimates on the number of critical points for any smooth function on the manifold. While the Lusternik-Schnirelmann theory was mainly used in topology and analysis, it had far-reaching consequences in geometry as well, such as the well-known results on existence of multiple closed geodesics on manifolds.

We shall now introduce the concept of a category in the sense of critical point theory  on an abstract simplicial complex $(V,\K)$ at level $m$. We recall the classical Lusternik-Schnirelman category  (see \cite{LS34,CLP03,FMV15}) of a closed subset $S\subset |S_\K|$:
\begin{align*}
\mathrm{cat}(S):=\min\{k\in \mathbb{N}^+:&\exists k+1 \text{  closed subsets }U_0,U_1,\cdots,U_k \text{  contractible in\footnotemark[8] }|S_\K|
\\&\text{ and  }\cup_{i=0}^k U_i\supset S\},
\end{align*}
where a subset   $U\subset |S_\K|$ is {\sl contractible in $|S_\K|$}  if there exists a continuous map $\eta:[0,1]\times |S_\K|\to |S_\K|$ such that $\eta(0,\cdot)=\mathrm{id}_{|S_\K|}$ and $\eta(1,U)=\text{ one-point set}$.  
With the aid of the Lusternik-Schnirelman  category, we introduce a   families of subsets of $\K$:
\footnotetext[8]{The references \cite{CLP03,FMV15} refer to the sets $U_0,\cdots, U_k$ as {\sl categorical}  sets, i.e.,  the  inclusion map $U_i\hookrightarrow |S_\K|$ is null-homotopic, $\forall i$.}
\begin{equation}\label{eq:LSC-K}
\mathsf{LSC}_m(\K)=\{L\subset\K:\;\mathrm{cat}(|S_\K(L)|)\ge m\},\;m=0,1,\cdots,
\end{equation}
where $S_\K(L)$ is the induced subcomplex of $S_\K$ on $L$. Note that this is a family of subsets of $\power(\K)$ (not $\power(V)$!). \; Similarly,
\begin{equation}\label{eq:LSC-S_K}
\mathsf{LSC}_m(|S_\K|)=\{S\subset|S_\K|:\;\mathrm{cat}(S)\ge m,\,S\text{ is closed}\}
\end{equation}

We are now ready to establish a Lusternik-Schnirelman category theorem relating a discrete Morse function and its Lov\'asz extension:
\begin{theorem}[L-S category theorem for a discrete Morse function and its Lov\'asz extension]
	\label{thm:LS-category-discrete-Morse-Lovasz}
	Let $f:\K\to\R$ be a  discrete Morse function. Then we have a sequence of critical values:
	$$
	\min\limits_{L\in\mathsf{LSC}_m(\K)}\max\limits_{\sigma\in L} f(\sigma)=\inf\limits_{S\in \mathsf{LSC}_m(|S_\K|)}\sup\limits_{\vec x\in S} f^L(\vec x),\;\;\;m=0,1,\cdots,\dim\K.
	$$
\end{theorem}

\begin{proof}
Without loss of generality, we may assume that $f:\K\to\R$ is  an injective   discrete Morse function. This is because any discrete Morse function can easily be
replaced by a 1-1 discrete Morse function that has the same induced gradient
vector field (up to a reasonable notion of equivalence).   
For any $S\in \mathsf{LSC}_m(|S_\K|)$, $f^L$ achieves a  maximum on $S$ at some point $\vec s$, that is,  $f^L(\vec s)=\sup\limits_{\vec x\in S} f^L(\vec x)$. If $\vec s$ does not belong to the vertex set of $|S_\K|$, then $\vec s$ is not an inner point of $S$ according to the definition of $f^L$. So $\vec s\in\partial S\setminus \mathrm{Vertex}(|S_\K|)$, and thus we can take a small perturbation $S'$ of $S$ such that $S'\in \mathsf{LSC}_m(|S_\K|)$ and $\sup f^L(S')<f(\vec s)=\sup f^L(S)$. Therefore, we only need to consider such $S$ with the property that $\max_{\vec x\in S} f^L(\vec x)$ is achieved at some vertex $\vec v$  of $|S_\K|$, i.e., 
$f^L(\vec v)\ge f^L(\vec x)$, $\forall \vec x\in S$. 
Consider the sublevel set $\{f^L\le f(\vec v)\}$. It is clear that $\mathrm{cat}(\{f^L\le f(\vec v)\})\ge \mathrm{cat}(S)\ge m$ and $\max f^L(\{f^L\le f(\vec v)\})=f(\vec v)=\max f^L(S)$.

\vspace{0.06cm}

\textbf{Claim}. $\mathrm{cat}(\{f^L\le a\})=\mathrm{cat}(S_\K|_{\{\sigma\in \K:f(\sigma)\le a\}})$, where $S_\K|_{\{\sigma\in \K:f(\sigma)\le a\}}$ is the induced closed subcomplex of $S_\K$ on the vertices $\{\sigma\in \K:f(\sigma)\le a\}$ of $S_\K$.

\vspace{0.03cm}

\textbf{Proof}. In fact, by Lemma \ref{lemma:Kuhnel}, there is a homotopy equivalence between $\{f^L\le a\}$ and $S_\K|_{\{\sigma\in \K:f(\sigma)\le a\}}$. 

\vspace{0.06cm}

By the above claim, we establish the following identities
\begin{align*}
\inf\limits_{S\in \mathsf{LSC}_m(|S_\K|)}\sup\limits_{\vec x\in S} f^L(\vec x)&=\inf\limits_{a\in\R\text{ s.t. }\{f^L\le a\}\in \mathsf{LSC}_m(|S_\K|)}\sup\limits_{\vec x\in \{f^L\le a\}} f^L(\vec x)
\\&=\min\limits_{a\in\R\text{ s.t. }S_\K|_{\{\sigma\in \K:f(\sigma)\le a\}}\in\mathsf{LSC}_m(\K)\;\;}\max\limits_{\;\;\sigma\in S_\K|_{\{\sigma\in \K:f(\sigma)\le a\}}}f(\sigma)
\\&=\min\limits_{L\in\mathsf{LSC}_m(\K)}\max\limits_{\sigma\in L} f(\sigma).
\end{align*}
\end{proof}

We point out that our concept of discrete Lusternik-Schnirelman category for abstract simplicial complexes is different from that of  Definition 4.3 in \cite{DENV19}. 

\section{Discrete Morse theory on complex-like  hypergraphs}
\label{sec:Morse-hypergraph}
In the preceding, we have established a correspondence between the discrete Morse theory on a simplicial complex $\K$ with vertex set $V$ and the continuous Morse theory on the associated order complex $S_\K$. Since the order complex $S_\E$ is still a simplicial complex when $\E$ is only a hypergraph with vertex set $V$, we can use the continuous Morse theory on that complex to define a discrete Morse theory on $\E$. That is what we shall now do.

A {\it hypergraph} is a pair $(V,\E)$ with 
 $\E\subset \power(V)$. In other words, $\E$ is a general set family on $V$.  We study  the combinatorial structure of a hypergraph from a topological perspective.
 
Since the philosophy of Morse theory is to understand the topology  by functions, we should  make clear what is the  topology on a hypergraph. Section \ref{sec:topology-hypergraph} indicates that it is nice to work on the order complex. Moreover, Section \ref{sec:DM-hypergraph-complex-like} reveals that we should concentrate on a hypergraph which looks like a simplicial complex if we want to establish Theorem  \ref{Mainthm:discrete-Morse} on hypergraphs.
 
\subsection{Topologies on hypergraphs} \label{sec:topology-hypergraph}


There are several natural approaches available to define a topology on a  hypergraph. 

The concept of finite topology is  introduced and studied by Alexandrov
\cite{Alexandrov}, and later by Stong from the  perspective of  algebraic topology  \cite{Stong66}.  

There are several ways to endow a finite hypergraph $(V,\E)$  with a topology
 on its hyperedge set $\E$.

\begin{defn}\label{def:down-topo}
The (down) finite topology $\T$ on $\E$ is generated by the base $\{U_e\}_{e\in\E}$, where $U_e=\{e'\in\E:e'\subset e\}$. 
\end{defn} 

\begin{defn}
The (up) finite topology $\T'$ on $\E$ is generated by the base $\{U_e^{up}\}_{e\in\E}$, where $U_e^{up}=\{e'\in\E:e'\supset e\}$. 
\end{defn}

 \begin{defn}
 The \textbf{order complex} of $\E$ denoted by
$$S_\E:=\{\C\subset\E: \C\text{ is a chain}\}$$
collects all inclusion chains in $\E$.
It is clear that $S_\E$ is a simplicial complex with the vertex set
$\E$. Define the  geometric realization of $S_\E$ by
$$|S_\E|= \bigcup_{\C\in S_\E}\conv(\vec 1_{e}:e\in \C).$$
\end{defn}

associated simplicial complex \cite{PL91}. 
 \begin{defn}\label{def:associated-complex}
The {\sl associated simplicial complex} $(V,\K_{\E})$  
is the smallest simplicial complex $\K_\E$ containing $\E$. 
Each hyperedge $e$ corresponds to an open simplex $|e|$ in the geometric realization $|\K_{\E}|$. 
The geometric realization $|\E|$ is then  defined as  $\bigcup\limits_{e\in \E}|e|$ in the geometric simplicial complex $|\K_{\E}|$.  
\end{defn}
As a subset of $|\K_{\E}|$, the geometric realization (or underlying space)
$|\E|$ may be neither closed nor open (see Examples \ref{ex:1} and
\ref{ex:2}). In contrast, the order complex $|S_\E|$ is closed.

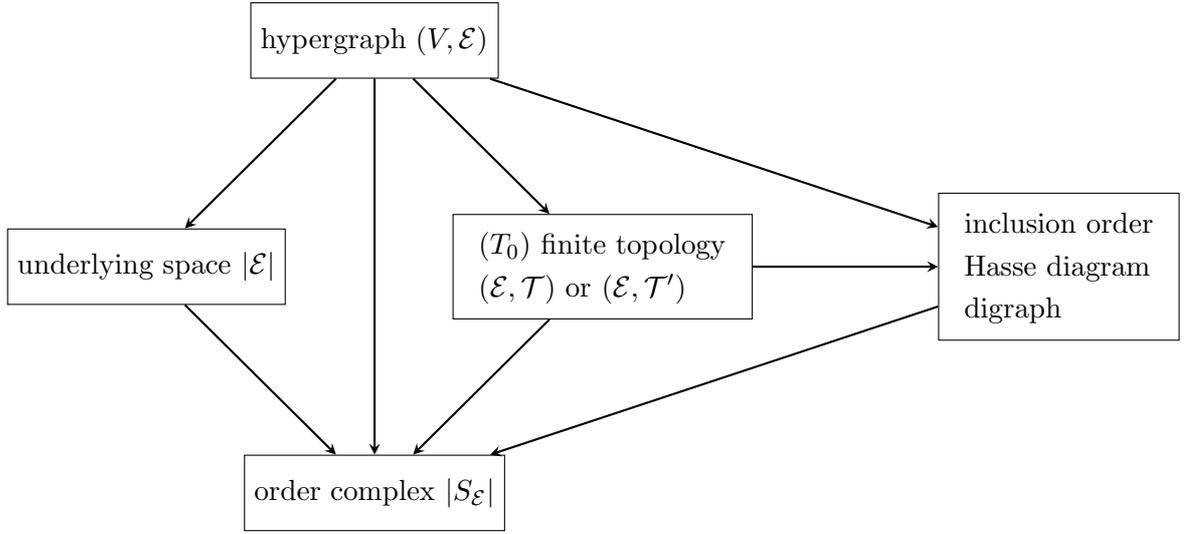
\begin{figure}[H]
\begin{tikzpicture}[node distance=6cm]
\node (h) [process] {hypergraph $(V,\E)$};
\node (g) [process, below of=h, xshift=-3cm, yshift=3cm] {underlying space $|\E|$};
\node (p) [process, below of=h, xshift=9cm, yshift=3cm] {
\begin{tabular}{ll}
inclusion order \\ Hasse diagram \\ digraph 
\end{tabular}};
\node (f) [process, below of=h, xshift=3cm, yshift=3cm] {\begin{tabular}{ll}
($T_0$) finite topology \\ $(\E,\T)$ or $(\E,\T')$
\end{tabular}};
\node (o) [process, below of=h] {order complex $|S_\E|$};
\draw [arrow](h) --node[anchor=south]{\small  }  (g);
\draw [arrow](h) --node[anchor=south]{\small  }  (p);
\draw [arrow](h) --node[anchor=south]{\small  }  (f);
\draw [arrow](h) --node[anchor=south]{\small  }  (o);
\draw [arrow](f) --node[anchor=south]{\small  }  (o);
\draw [arrow](g) --node[anchor=south]{\small  }  (o);
\draw [arrow](p) --node[anchor=south]{\small  }  (o);
\draw [arrow](f) --node[anchor=south]{\small  }  (p);
\end{tikzpicture}
\caption{\label{Fig:topologoes-on-E} The topologies on  hypergraphs in this paper. }
\end{figure}
In this section,  
we provide some detailed descriptions on the relations among these objects. 

\begin{defn}(Section 1.4 in \cite{Barmak11})
A weak
homotopy equivalence between two topological spaces $X$ and $Y$ is a continuous map $X\to Y$ (or $Y\to X$) which induces isomorphisms in all homotopy groups.

Two topological spaces $X$ and $Y$ are {\sl weakly homotopy equivalent} (denoted by $\mathop{\simeq}\limits^{weak}$) if there exists a finite sequence of topological spaces $X=X_0,X_1,\cdots,X_n=Y$ such that there are weak
homotopy equivalences $X_i\to X_{i+1}$ (or $ X_{i+1}\to X_i$) for every $0\le i\le n-1$. 
\end{defn}
It is interesting that the above four topologies are pairwise distinct, but they are weakly homotopy equivalent. In detail, we have:

\begin{pro}\label{pro:topology-on-hyperedges}
The topologies as stated above  are (weakly) homotopy equivalent, i.e.,
 $$(\E,\T)\mathop{\simeq}\limits^{weak}(\E,\T')\mathop{\simeq}\limits^{weak}|\E|\simeq|S_{\E}|.$$
\end{pro}
\begin{proof}
The proof of $(\E,\T)\mathop{\simeq}\limits^{weak}|S_{\E}|$ is essentially established in  \cite{McCord66} for finite topologies (see also \cite{Barmak11}). 
Consider a new hypergraph $(V,\E')$ defined by $\E'=\{ e':\,
e\in\E\}$, where $e'=V\setminus e$ is the complement of $e$ in $V$.  Then $(\E,\T')\mathop{\simeq}\limits^{weak}|S_{\E'}|$. Since the order complex associated with an order is homeomorphic to the order complex associated with its inverse order (see \cite{Barmak11}), we have $|S_{\E'}|=|S_{\E}|$.

To prove $|\E|\simeq|S_{\E}|$, we only need to show that $|S_{\E}|$ is a strong deformation retract 
of $|\E|$.

Let $\K$ be the smallest simplicial complex containing $\E$, and note the following facts:
\begin{enumerate}
\item $|\E|=|\K|\setminus \bigsqcup_{e\in \K\setminus\E} |e|=\bigcap_{e\in \K\setminus\E}(|\K|\setminus|e|) $ and $|S_\E|=|\K|\setminus \bigcup_{e\in \K\setminus\E}|\mathrm{star}_{S_\K}(e)|=\bigcap_{e\in \K\setminus\E}(|\K|\setminus|\mathrm{star}_{S_\K}(e)|)$, where $|\mathrm{star}_{S_\K}(e)|$ is the relatively open geometric realization of the star neighborhood of $e$ in $|S_\K|$. 
\item $|\mathrm{star}_{S_\E}(e)|
\simeq|e|\simeq \mathrm{pt}$ for any $e\in\E$; $|\mathrm{star}_{S_\K}(e)|\simeq|\mathrm{star}_{\K}(e)|\simeq|e|\simeq \mathrm{pt}$  for any $e\in\K$.
\item There is a deformation $\eta_e$ from $|\mathrm{star}_{S_\K}(e)|\setminus|e|$ to $ |\mathrm{link}_{S_\K}(e)|$, for any $e$. 
\item There is a natural  deformation $\eta_e':(|\K|\setminus|e|)\times[0,1]\to |\K|\setminus|e|$ with $\eta_e'(|\K|\setminus|e|,1)=|\K|\setminus |\mathrm{star}_{S_\K}(e)|$, for any $e$. 
\item Let $\K\setminus\E=\{e^1,\cdots,e^k\}$. Then the composition of the deformations $\eta_{e^1}'$, $\cdots$, $\eta_{e^k}'$,  gives a deformation from $|\E|$ to $|S_{\E}|$,  according to 
\begin{align*}
\eta'_{e^1}(\eta'_{e^2}(\cdots(\eta'_{e^k}(|\E|,1),\cdots),1),1)&=\eta'_{e^1}(\eta'_{e^2}(\cdots(\eta'_{e^k}(\bigcap_{i=1}^k(|\K|\setminus|e^i|),1),\cdots),1),1)
\\&=\bigcap_{i=1}^k(|\K|\setminus \mathrm{star}_{S_\K}(e^i))=|S_\E|.
\end{align*}
\end{enumerate}
The proof is completed. 
\end{proof}

\begin{example}\label{ex:1}
Let $V=\{1,2,3,4\}$ and $\E=\{\{1\},\{2\},\{3\},\{4\},\{1,2,3\},\{1,3,4\}\}$. We show the geometric realization $|\E|$ and the order complex $|S_\E|$ in the left picture in Remark \ref{remark:Nerve} below.
\end{example}
\begin{example}\label{ex:2}
Let $V=\{1,2,3,4\}$ and $\E=\{\{1\},\{2\},\{4\},\{1,2,3\},\{1,3,4\}\}$.  The geometric realization $|\E|$ and the order complex $|S_\E|$ are displayed  on the right  picture in Remark \ref{remark:Nerve}.
\end{example}

\begin{remark}\label{remark:Nerve}
The conclusion $|\E|\simeq|S_{\E}|$ in Proposition \ref{pro:topology-on-hyperedges} can be regarded as a special Nerve Theorem for  \v{C}ech-type  complexes. Indeed, we can see $|S_{\E}|$ as the nerve of $|\E|$ in  a certain sense. To visually understand this point, we draw the nerve complex of a special cover of the underlying space $|\E|$ for Examples \ref{ex:1} and \ref{ex:2}, respectively:
\begin{center}
\begin{tikzpicture}[scale=1.5]
\draw[dashed] (0,0)--(0,2);
\draw[dashed] (1.732,1)--(0,2);
\draw[dashed] (1.732,1)--(0,0);
\draw[dashed] (-1.732,1)--(0,2);
\draw[dashed] (-1.732,1)--(0,0);
\draw[blue] (0,2)--(0.577,1)--(0,0)--(-0.577,1)--(0,2);
\draw[blue] (0.577,1)--(1.732,1);
\draw[blue] (-0.577,1)--(-1.732,1);
\filldraw[black,fill opacity=0.15,line width=0,draw =red!0]  (0,0)--(0,2)--(1.732,1);
\filldraw[black,fill opacity=0.12,line width=0,draw =red!0]  (0,0)--(0,2)--(-1.732,1);
\node (1) at  (0,-0.2) {$1$};\node (1) at  (0,0) {$\bullet$};
\node (2) at  (1.782,1.2) {$2$};\node (2) at  (1.732,1) {$\bullet$};
\node (3) at  (0,2.2) {$3$};\node (3) at  (0,2) {$\bullet$};
\node (4) at  (-1.782,1.2) {$4$};\node (4) at  (-1.732,1) {$\bullet$};
\node (123) at  (0.6,1.2) {$123$};\node (123) at  (0.577,1) {{\color{blue}$\bullet$}};
\node (134) at  (-0.6,1.2) {$134$};\node (134) at  (-0.577,1) {{\color{blue}$\bullet$}};
\end{tikzpicture}\;\;\;\;\;\;
\begin{tikzpicture}[scale=1.5]
\draw[dashed] (0,0)--(0,2);
\draw[dashed] (1.732,1)--(0,2);
\draw[dashed] (1.732,1)--(0,0);
\draw[dashed] (-1.732,1)--(0,2);
\draw[dashed] (-1.732,1)--(0,0);
\draw[blue] (0.577,1)--(0,0)--(-0.577,1);
\draw[blue] (0.577,1)--(1.732,1);
\draw[blue] (-0.577,1)--(-1.732,1);
\filldraw[black,fill opacity=0.15,line width=0,draw =red!0]  (0,0)--(0,2)--(1.732,1);
\filldraw[black,fill opacity=0.12,line width=0,draw =red!0]  (0,0)--(0,2)--(-1.732,1);
\node (1) at  (0,-0.2) {$1$};\node (1) at  (0,0) {$\bullet$};
\node (2) at  (1.782,1.2) {$2$};\node (2) at  (1.732,1) {$\bullet$};
\node (4) at  (-1.782,1.2) {$4$};\node (4) at  (-1.732,1) {$\bullet$};
\node (123) at  (0.6,1.2) {$123$};\node (123) at  (0.577,1) {{\color{blue}$\bullet$}};
\node (134) at  (-0.6,1.2) {$134$};\node (134) at  (-0.577,1) {{\color{blue}$\bullet$}};
\end{tikzpicture}
\end{center}
As in the case of simplicial complexes, the vertices of $S_\E$
  correspond to the barycenters of the hyperedges in $\E$. And the deformation process from  $|\E|$ to $|S_{\E}|$ can be  visualized as the following picture (for Examples \ref{ex:1} and \ref{ex:2}, respectively):
  
  \begin{center}
\begin{tikzpicture}[scale=1.5]
\draw[dashed] (1.732,1) arc(270:210:2);
\draw[dashed] (1.732,1) arc(90:150:2);
\draw[dashed] (-1.732,1)arc(270:330:2);
\draw[dashed] (-1.732,1)arc(90:30:2);
\draw[dashed] (0,2) arc(150:210:2);
\draw[dashed] (0,2) arc(30:-30:2);
\draw[blue] (0,2)--(0.577,1)--(0,0)--(-0.577,1)--(0,2);
\draw[blue] (0.577,1)--(1.732,1);
\draw[blue] (-0.577,1)--(-1.732,1);
\filldraw[black,fill opacity=0.09,line width=0,draw =red!0]  (0,2) arc(30:-30:2) (0,0) -- (0.577,1) --(0,2);
\filldraw[black,fill opacity=0.09,line width=0,draw =red!0] (0,0) arc(150:90:2) (1.732,1) -- (0.577,1) --(0,0); \filldraw[black,fill opacity=0.09,line width=0,draw =red!0] (1.732,1) arc(270:210:2) (0,2) -- (0.577,1) --(1.732,1);
\filldraw[black,fill opacity=0.06,line width=0,draw =red!0]  (0,2) arc(150:210:2) (0,0) -- (-0.577,1) --(0,2);
\filldraw[black,fill opacity=0.06,line width=0,draw =red!0] (0,0) arc(30:90:2) (-1.732,1) -- (-0.577,1) --(0,0); \filldraw[black,fill opacity=0.06,line width=0,draw =red!0] (-1.732,1) arc(270:330:2) (0,2) -- (-0.577,1) --(-1.732,1);
\node (1) at  (0,-0.2) {$1$};\node (1) at  (0,0) {$\bullet$};
\node (2) at  (1.782,1.2) {$2$};\node (2) at  (1.732,1) {$\bullet$};
\node (3) at  (0,2.2) {$3$};\node (3) at  (0,2) {$\bullet$};
\node (4) at  (-1.782,1.2) {$4$};\node (4) at  (-1.732,1) {$\bullet$};
\node (123) at  (0.6,1.2) { $123$};\node (123) at  (0.577,1) {{\color{blue}$\bullet$}};
\node (134) at  (-0.6,1.2) {$134$};\node (134) at  (-0.577,1) {{\color{blue}$\bullet$}};
\end{tikzpicture}\;\;\;\;\;\;\begin{tikzpicture}[scale=1.5]
\draw[dotted] (0,0) to[out=90,in=250] (0.288,1.5) to[out=0,in=150] (1.732,1);
\draw[dashed] (1.732,1)to[out=180,in=60](0,0);
\draw[dotted] (0,0) to[out=90,in=290] (-0.288,1.5) to[out=0,in=30] (-1.732,1);
\draw[dashed] (-1.732,1) to[out=0,in=120](0,0);
\draw[blue] (0.577,1)--(0,0)--(-0.577,1);
\draw[blue] (0.577,1)--(1.732,1);
\draw[blue] (-0.577,1)--(-1.732,1);
\filldraw[black,fill opacity=0.09,line width=0,draw =red!0]  (0,0) to[out=90,in=250] (0.288,1.5) to[out=0,in=150] (1.732,1) to[out=180,in=60](0,0);
\filldraw[black,fill opacity=0.06,line width=0,draw =red!0]  (0,0) to[out=90,in=290] (-0.288,1.5) to[out=180,in=30] (-1.732,1) to[out=0,in=120](0,0);
\node (1) at  (0,-0.2) {$1$};\node (1) at  (0,0) {$\bullet$};
\node (2) at  (1.782,1.2) {$2$};\node (2) at  (1.732,1) {$\bullet$};
\node (4) at  (-1.782,1.2) {$4$};\node (4) at  (-1.732,1) {$\bullet$};
\node (123) at  (0.6,1.2) {$123$};\node (123) at  (0.577,1) {{\color{blue}$\bullet$}};
\node (134) at  (-0.6,1.2) {$134$};\node (134) at  (-0.577,1) {{\color{blue}$\bullet$}};
\end{tikzpicture}
  \end{center}
\end{remark}

Proposition \ref{pro:topology-on-hyperedges} allows us to define the homology of a hypergraph by 
 $$H(\E):=H(|\E|)\cong H(\E_\T) \cong H(\E_{\T'})\cong H(|S_\E|)$$
 where $\E_\T$ (resp. $\E_{\T'}$) indicates the hypergraph endowed with the down (resp. up) finite topology (see Definition  \ref{def:down-topo}). In fact, since  
 McCord's map 
$|S_\E|\ni \sum_{e\in\C} t_e\vec 1_e\mapsto \bigcup\limits_{e\in\C} e\in\E$ is
a weak homotopy equivalence \cite{McCord66}, it induces an isomorphism
$H(|S_\E|)\to H(\E_\T)$. And the deformation from $|\E|$ to $|S_\E|$ defined in the proof of Proposition \ref{pro:topology-on-hyperedges} induces an isomorphism $H(|\E|)\to H(|S_\E|) $.

Furthermore, by  the excision theorem and other  basic results of
  homology theory, we have
\begin{pro}[Homology properties]\label{pro:relative-homology}
For $\A\subset\B\subset \E$ and $\A'\subset\B'\subset\E$ with $\B\setminus \A=\B'\setminus \A'$, we can regard $(V,\A)$, $(V,\B)$, $(V,\A')$ and $(V,\B')$ as the sub-hypergraphs of $(V,\E)$. Then we have
\begin{align*}
H(\B,\A):&=H(|\B|,|\A|)\cong H(\B_\T,\A_\T)\cong H(\B_{\T'},\A_{\T'})\cong H(|S_\B|,|S_\A|)
\\& \cong H(|S_{\B'}|,|S_{\A'}|)\cong H(\B'_{\T'},\A'_{\T'})\cong H(\B'_{\T},\A'_{\T})\cong H(|\B'|,|\A'|):= H(\B',\A').
\end{align*}
\end{pro}


Propositions \ref{pro:topology-on-hyperedges} and \ref{pro:relative-homology}
suggest to consider the order complex instead of other complexes like
the \v{C}ech  and the Vietoris-Rips complex with respect to $\E$, which do not include the precise topological data of $\E$ (see Remarks \ref{remark:cech} and \ref{remark:VR} for some comments).

\begin{remark}\label{remark:cech}
The \textbf{\v{C}ech complex} for $\E$ is the simplicial complex with  vertex set $\E$ and face set $\mathcal{N}_\E:=\{E'\subset \E:\bigcap_{e\in E'}e\ne\varnothing\}$.

However, 
it can be verified that $|\mathcal{N}_\E|\simeq |\K_\E|$. If $(V,\E)$ is a
simplicial complex, then $|\mathcal{N}_\E|\simeq |\K_\E|= |\E|\simeq
|S_\E|$. But if $(V,\E)$ is not a simplicial complex, $|\mathcal{N}_\E|$ may
not be homotopic to $|\E|\simeq |S_\E|$. Due to this reason, we shall work on
the order complex instead of  the  C\v{e}ch complex for a hypergraph.
\end{remark}

\begin{remark}\label{remark:VR}
The \textbf{Vietoris-Rips complex} for $\E$ is the simpicial complex with  vertex set $\E$ and face set $\mathcal{VR}_\E:=\{E'\subset \E:e'\cap e\ne\varnothing,\forall e',e\in E'\}$. Also, $|\mathcal{VR}_\E|\not \simeq |\E|$ in general.

According to the inclusion relation $S_\E\subset \mathcal{N}_\E\subset
\mathcal{VR}_\E$ and Remark \ref{remark:cech}, it is better to work with the order complex $|S_\E|$.
\end{remark}

\begin{remark}
The embedded homology of a hypergraph $\E$ introduced in \cite{BLRW19} is different from $H(\E)$. As an  embedded homology may not be the homology of any simplicial
complex \cite{GWXW21}, we would like to work on the geometric realization $|\E|$ and the order complex $|S_\E|$, which are more geometrically intuitive.

\end{remark}
\subsection{Discrete Morse function on  special hypergraphs}
\label{sec:DM-hypergraph-complex-like}

To establish a More theory on hypergraphs, we can first work on a complex-like hypergraph, whose combinatorial structure and topological structure  are   similar to a simplicial complex. The next definition for  Morse functions on  hypergraphs follows  verbatim from the original definition of  Morse functions on  simplicial complexes by Forman.

\begin{defn}
An edge pair $(e', e)$ is called {\sl sequential} if $e'\subsetneqq e$ and there is no other $e''$ with $e'\subsetneqq e''\subsetneqq e$. A function  $f:\E\to \R$ is a {\sl simple discrete Morse function} if it has the property that for any $e\in\E$, $\# \{\text{sequential pair }(e', e): f(e')\ge f(e)\}\le 1$ and $\#\{\text{sequential pair }(e,\tilde{e}): f(e)\ge f(\tilde{e})\}\le 1$. An edge $e$ is called a critical point of a simple discrete Morse function $f$ if $\{\text{sequential pair }(e', e): f(e')\ge f(e)\}=\varnothing=\{\text{sequential pair }(e,\tilde{e}): f(e)\ge f(\tilde{e})\}$. We say that $e$ has height $k$ if there are at most $k$ edges, $e^1,\cdots,e^k$, in a chain of the form  $e^1\subsetneqq e^2\subsetneqq\cdots\subsetneqq e^k\subsetneqq e$. A critical point $e$ of $f$ has {\sl index} $k$ if the height of $e$ is $k$.
\end{defn}
We have a preliminary result for special hypergraphs and the corresponding typical functions, which is a  slight 
generalization of Forman's discrete Morse theory.
\begin{theorem}\label{thm:complex-like-hypergraph}
For a finite hypergraph $(V,\E)$, assume that $\E$ has the properties that the geometric realization $|\{e'\in\E:e'\subsetneqq e \}|$ is homotopic to a sphere for any $e$, and the geometric realization $|\{e''\in\E:e''\subset e,\,e''\not\in\{e',e\} \}|$ is contractible for any sequential edge pair $(e',e)$. Let $f:\E\to \R$ be a simple discrete Morse function with  exactly one  critical point of index $k$. Then the geometric realization $|\E|$ is homotopy equivalent to a CW-complex with one $k$-cell.
\end{theorem}

\begin{defn}[complex-like hypergraph]
A finite hypergraph satisfying the conditions in Theorem \ref{thm:complex-like-hypergraph} is called a {\sl complex-like hypergraph.}
\end{defn}

It is clear that every simplicial complex is a complex-like hypergraph. 
Example \ref{ex:2} shows a complex-like hypergraph which is not a simplicial complex, while Example \ref{ex:1} shows a hypergraph which is not complex-like. One may find that the hypergraph in Example \ref{ex:2} and the simplicial complex $\K=\{\{1\},\{2\},\{3\},\{1,2\},\{2,3\}\}$ have the same order complex. However, there exists a complex-like  hypergraph whose order complex doesn't agree with any order complex on simplicial complexes (see Example \ref{exam:complex-like-not-simplicial}). So, the phrase `complex-like' essentially refers to `like a simplicial complex in topology'.

The complex-like condition  is technical but natural,  ensuring the validity  of  Lemmas \ref{lemma:link-homotopy}, \ref{lemma:sub-level-critial} and \ref{lemma:weak-slope-metric}. In the sequel, all the results regarding Morse theory on a hypergraph require the  complex-like condition.  Actually, Example \ref{exam:general-hypergral-fail-Morse} indicates that the main theorems in this paper do not hold for general hypergraphs.

We then define  the  Lov\'asz extension restricted on  $|S_{\E}|$ 
as follows. 

According to Proposition \ref{pro:setfamily-order-complex}, the Lov\'asz extension $f^L$ is well-defined on $|S_{\E}|$ for any $f:\E\to\R$. Indeed, the feasible domain of  $f^L$ is 
$$\D_\E=
\begin{cases}
\bigcup\limits_{t\ge0}t|S_\E|\;\;\;\;\;\;\;\;\;\;\;=\{t\vec x:t\ge0,\vec x\in|S_\E|\}\subset \R_{\ge0}^V,&\text{ if }V\not\in \E,\\
\bigcup\limits_{t\ge0}t|S_\E|+\R\vec1_V=\{t\vec x:t\ge0,\vec x\in|S_\E|\}+\mathrm{span}(\vec 1_V),&\text{ if }V\in \E.
\end{cases}
$$
In any case,  $f^L$ is well-defined on the polyhedral cone $\bigcup\limits_{t\ge0}t|S_\E|$. We shall restrict the Lov\'asz extension  $f^L$   
on $|S_{\E}|$.

Let $$\mathsf{LSC}_m(\E)=\{E'\subset\E:\;\mathrm{cat}(|S_\E(E')|)\ge m\}$$
where $S_\E(E')$ is the induced subcomplex of $S_\E$ on $E'$.  
Then, Theorems \ref{thm:critical-discrete-Morse}, \ref{thm:Morse-vector-f} and \ref{thm:LS-category-discrete-Morse-Lovasz} can also be generalized to this setting of hypergraphs:
\begin{theorem}
	\label{Mainthm:discrete-Morse-hyper}
	For a hypergraph $(V,\E)$ under the assumptions of Theorem \ref{thm:complex-like-hypergraph}, let $f:\E\to\R$ be an injective discrete Morse function. Then the following conditions are equivalent:
	\begin{enumerate}[(1)]
		\item $e$ is a critical point of $f$ with index $i$;
		\item $\vec1_e$ is a critical point of $f^L|_{|S_\E|}$ with index $i$ in the sense of weak slope (metric Morse theory);
		\item  $\vec1_e$ is a critical point of $f^L|_{|S_\E|}$ with index $i$ in the sense of K\"uhnel (PL Morse theory);
		\item  $\vec1_e$ is a Morse critical point of $f^L|_{|S_\E|}$ with index $i$ in the sense of topological Morse theory.
	\end{enumerate}

	Moreover, the discrete Morse vector $(n_0,n_1,\cdots,n_d)$,  representing the number $n_i$ of critical points with index $i$, of $f$ coincides with the  continuous Morse vector of $f^L|_{|S_\E|}$.
	
	Moreover, the Lusternik-Schnirelmann category theorem is preserved  under Lov\'asz extension:
	$$
	\min\limits_{E'\in\mathsf{LSC}_m(\E)}\max\limits_{e\in E'} f(e)=\inf\limits_{S\in \mathsf{LSC}_m(|S_\E|)}\sup\limits_{\vec x\in S} f^L(\vec x),
	$$
\end{theorem}

The proof of Theorem \ref{Mainthm:discrete-Morse-hyper} follows verbatim  the proof of Theorem \ref{Mainthm:discrete-Morse} in Section \ref{sec:discrete-Morse}, and thus we omit it.  

\begin{example}\label{exam:complex-like-not-simplicial}
The setting of Theorem \ref{Mainthm:discrete-Morse-hyper} is a little wider than Theorem \ref{Mainthm:discrete-Morse}. For example, taking $V=\{1,2,3,4\}$ and $\E=\{\{1\},\{2\},\{1,2,3\},\{1,2,4\},\{1,2,3,4\}\}$. It is clear that $\E$ is not a simplicial complex. Moreover, its order complex $S_\E$ cannot be an order complex of any simplicial complex. However, one can check that $(V,\E)$ is complex-like, and both Theorem \ref{thm:complex-like-hypergraph} and Theorem \ref{Mainthm:discrete-Morse-hyper} are valid for $(V,\E)$.
\end{example}


Theorem \ref{Mainthm:discrete-Morse-hyper}  doesn't hold for a general hypergraph. 
In fact, we need the complex-like condition  to obtain appropriate  analogs of Lemmas \ref{lemma:link-homotopy}, \ref{lemma:sub-level-critial} and \ref{lemma:weak-slope-metric}.  
\begin{example}\label{exam:general-hypergral-fail-Morse}
Let $V=\{1,\cdots,n\}$ and $\E=\{\{1\},\{1,2\},\cdots,\{1,\cdots,k\},\cdots,\{1,\cdots,n\}\}$. Then 
the function $f:\E\to\R$ defined by  $f(\{1,\cdots,k\})=k$ is a discrete Morse function, and every $e\in\E$ is a critical point of $f$, but every $\vec 1_e$ with $\#e\not\in\{1,n\}$ 
 is a regular point of $f^L|_{|S_\E|}$. Therefore,  Theorem \ref{Mainthm:discrete-Morse-hyper}  fails for such a hypergraph.
 \end{example}
 
 It is expected that a more general Morse theory of hypergraphs and more applications will be developed in the future.  
 The key idea is that the definition of  critical points of a general function $f$ on $\E$ is  translated into the PL critical point theory of its restricted Lov\'asz extension $f^L|_{|S_\E|}$. 

\begin{defn}
\label{defn:critical-hyper}
Given a finite hypergraph $(V,\E)$ and a function $f:\E\to \R$,
we say that $e\in\E$ is a critical point of $f$ if $\vec1_e$ is a critical point of $f^L|_{|S_\E|}$ in the sense of PL Morse theory.
\end{defn}

\section{Conclusions and Discussions 
}

In this paper, we present a systematic approach on  constructing  Lov\'asz extensions of discrete Morse functions, with which we build the foundational theory to translate one Morse
theory to another. 
Our dictionary on translating different Morse
theories is useful and has the potential to solve some problems in discrete Morse theory.   We can work on the barycentric  subdivision (or the order complex) of a simplicial complex to explore the critical simplexes.  Based on our approach, we  provide  a new
definition of a discrete Lusternik-Schnirelmann category. 

We also  propose a  general definition of critical points of a function on the edge set of a hypergraph, and this allows us to 
study the Morse theory on hypergraphs by employing the PL Morse theory on polyhedral structures.  This idea works well for   certain hypergraphs, e.g., the  complex-like hypergraphs in this paper, and it is possible to  get further results 
along this direction.

We leave the following two open problems for future research:
\begin{Question}
Can we modify a discrete Morse function $f$ on a  simplicial complex to a function with fewer critical simplices by exploiting the structure of $f^L$ to perform  handle-cancellations?
\end{Question}

\begin{Question}
Can we get a concise formulation of  the discrete Morse theory on hypergraphs regarding the definition of critical points  introduced in  Definition \ref{defn:critical-hyper}?
\end{Question}

Finally, we should point out that there are many other ways to introduce topological 
structures   on hypergraphs, for example following ideas related to   
independence 
complexes \cite{Emtander09} and  Hom complexes 
\cite{Dochtermann09,Kozlov06}.   Analogously to various techniques developed for  simplicial complexes, for special  hypergraphs or posets having good patterns, one could apply methods from analysis and homotopy theory   to study these discrete structures.   To better understand these structures, further exploration of the relationship between different topological structures on families of hypergraphs is needed. The frontier of research  in topological, geometrical, and dynamical  combinatorics 
has  the potential to provide new mathematical tools in data science.

\section*{Acknowledgement}
{\small
Dong Zhang would like to thank Professor Kung-Ching Chang for his long-term guidance, encouragement and support in mathematics.
}

{ \linespread{0.95} 
}

\section*{Appendix}
\begin{pro}\label{pro:setfamily-order-complex}
Given a set family $\E\subset\power(V)$ on a finite set $V$,
let $$S_\E:=\{\C\subset \E:\C\text{ is a chain under the inclusion relation}\}.$$ Then $(\E,S_\E)$ forms an abstract simplicial complex with   vertex set $\E$ and face set $S_\E$.

Let $|S_\E|$ be a geometric realization determined by the simplicial map  $e\mapsto \vec 1_{e}$, $\forall e\in \E$. Then
$$\D_\E=
\begin{cases}
\{t\vec x:t\ge0,\vec x\in|S_\E|\}\subset [0,\infty)^n,&\text{ if }V\not\in \E,\\
\{t\vec x:t\ge0,\vec x\in|S_\E|\}+\mathrm{span}(\vec 1_V),&\text{ if }V\in \E.
\end{cases}
$$
\end{pro}
\begin{proof} By the definition, $(\E,S_\E)$ is a simplicial complex with the facet set consisting of maximal chains under the inclusion relation. 

We determine the feasible domain $\D_\E$ of the Lov\'asz extension of any function $f:\E\to\R$  as follows.

\begin{enumerate}
\item If $V\in \E$, then  according to the definition of $\D_\E$,  
there holds
\begin{align*}
\D_\E&=\{\vec x\in\R^n:V^t(\vec x)\in \E,\,\forall t\in (-\infty,\max\vec x)\}
\\&=\left\{\vec x\in\R^n\left|\vec x=\sum_{i\ge0} t_i\vec 1_{e_i}\text{ with }V=e_0\supset e_1\supset e_2\supset\cdots, \,t_0\in\R,\,e_i \in \E, t_i\ge0, i\ge 1 \right.\right\}
\\&=\bigcup_{\text{maximal chain }\C\subset\E}\cone(\vec 1_{e}:e\in \C)+\mathrm{span}(\vec 1_V)
\\&=\{t\vec x:t\ge0,\vec x\in|S_\E|\}+\mathrm{span}(\vec 1_V).
\end{align*}
\item If $V\not\in \E$, 
we first show that one must assume  $\min_i x_i= 0$. 
Otherwise, 
$\min_i x_i\ne 0$, then the Lov\'asz extension has a term $\min_i x_if(V)$ which needs the data of $f(V)$, but $V\not\in \E$ and thus  it is impossible to get the value of $f(V)$. 
Therefore,  $\min_i x_i$ should be set  as $0$.  Similar to the above case, for $\vec x\in\R^n$ with $\min_i x_i= 0$,  
\begin{align*}
\D_\E&=\{\vec x\in\R^n:V^t(\vec x)\in \E,\,\forall t\in [0,\max\vec x)\}
\\&=\left\{\vec x\in\R^n\left|\vec x=\sum_{i\ge1} t_i\vec 1_{e_i}\text{ with }V
\ne e_1\supset e_2\supset\cdots, \,e_i \in \E, t_i\ge0, i\ge 1 \right.\right\}
\\&=\bigcup_{\text{maximal chain }\C\subset\E}\cone(\vec 1_{e}:e\in \C)
\\&=\{t\vec x:t\ge0,\vec x\in|S_\E|\}.
\end{align*}
\end{enumerate}
The proof is completed.
\end{proof}
\end{document}